\documentclass[12pt]{article}
\usepackage{amsthm}
\usepackage{amsmath}
\usepackage{amssymb,latexsym}
\usepackage[english]{babel}
\usepackage[dvips]{color,graphicx}
\DeclareGraphicsExtensions{.eps,.bmp,.pdf,.wmf,.xps}
\newtheorem{Th}{Theorem}
\newtheorem{Df}{Definition}
\newtheorem{Rem}{Remark}

\newtheorem{No}{Notation}
\newtheorem{Prp}{Proposition}
\newcounter{saveeqn}

\newcommand{\alpheqn}{\setcounter{saveeqn}{\value{equation}} \stepcounter{saveeqn} \setcounter{equation}{0}   \renewcommand{\theequation}{\mbox{\arabic{section}.\arabic{equation}}}}

\begin{document}

\begin{center}
{\textbf{\large Theory of two-parameter Markov chain with an application in  warranty study.}}
\vskip 3mm
\end{center}

\noindent
\vskip 3mm

\noindent \textsc{\'Alvaro Calvache}

\noindent \textit{\small Universidad de los Andes}  \, and \,  \textit{\small Universidad Pedag\'ogica y Tecnol\'ogica de Colombia. Colombia.}

\noindent \texttt{\footnotesize a.calvache402@uniandes.edu.co }

\vskip 5mm
\noindent \textsc{Arunachalam Viswanathan}

\noindent \textit{Universidad de los Andes. Bogot\'a, Colombia.}

\noindent \texttt{\footnotesize aviswana@uniandes.edu.co }
\vskip 3mm
\noindent \textbf{ABSTRACT:} In this paper we  present the classical results of Kolmogorov's backward and forward equations to the case of a two-parameter Markov process.  These equations relates the infinitesimal transition matrix of the two-parameter Markov process.  However, solving these equations is not possible and we require a numerical procedure. In this paper, we give an alternative method by use of double Laplace transform of the transition probability matrix and of the infinitesimal transition matrix of the process. An illustrative example is presented for the method proposed. In this example, we consider a two-parameter warranty model, in which a system can be any of these states: working, failure. We calculate the transition density matrix of these states and also the cost of the warranty for the proposed model.
\vskip 4mm

\alpheqn
\section{Introduction}
\setcounter{equation}{0}
One of the most important elements when we develop the theory about Markov processes with continuous time parameter and an enumerable space states, is the characterization that they have with their infinitesimal transition matrix. It is possible to find the matrix of transition's probabilities trough the system of differential equations of Kolmogorov. The main purpose of this work is to extend the concepts and results to the case of Markov processes with continuous two-dimensional parameters (which are usually the time and usage) and with an enumerable space states. For this purpose it is necessary to start introducing a concept similar to the infinitesimal transitions rates between states, which are defined in the case of Markov processes with a continuous-time parameter. In the case of continuous two dimensional parameters, the definition of the infinitesimal transition rates is that, in this case, they are not the derivatives of the transition's probabilities at time zero. Here it is required to work with the second derivatives of the transition's probabilities with respect to each of the parameters involved in the two-dimensional point $(0,0)$. Then, with these infinitesimal transition rates, new equations are proposed as backward and forward Kolmogorov equations.

The partial differential equations that are generated, in general, are not easy to solve, we use the double Laplace transform to present a simple result. This result however is not easy to invert. So, it requires the use of numerical methods for its inversion. After that, we present a result to find the distributions of the waiting region and these distributions can be used to calculate costs for a previously chosen warranty policy.

\section{Preliminaries}
\setcounter{equation}{0}
\begin{Df}[MCCTP]\normalfont
 A two-parameter stochastic process $\{X(t,u)\mid (t,u) \in [0,\infty )^2\}$ with discrete state space $S \subset \mathbb{Z}^+$, is a Markov Chain with continuous two-dimensional parameter (MCCTP), if for all $i,j,i_0,\ldots , i_k\in S$,\,for all  $t_0<t_1<\cdots <t_k<s<s+t$
 and for all $ u_0<u_1<\cdots <u_k<w<w+u$, the following equality holds:
\begin{multline}\label{c4MCCTP}\small
P \left\{X(t+s,u+w)=j \mid X(s,w)=i, X(t_0,u_0)=i_0, \ldots, X(t_k,u_k)=i_k  \right\}\,\, =\\
P \left\{X(t+s,u+w)=j \mid X(s,w)=i \right\},
\end{multline}
\end{Df}
\begin{Df}[MCHCTP]\normalfont
The Markov Chain with continuous two dimensional parameter $\left\{X(t,u)\mid (t,u) \in [0, \infty]^2 \right\}$  is homogeneous (MCHCTP), if
\begin{multline}\label{c4Homogeneous}
p_{ij}(t,u):=P\left\{X(s+t,w+u)=j\mid X(s,w)=i \right\}=\\P\left\{X(t,u)=j\mid X(0,0)=i \right\},\, \forall_{s,w \ge 0}
\end{multline}
\end{Df}

\begin{Rem}\normalfont
Note that
\begin{equation}\label{c4FProb}
0 \le p_{ij}(t,u) \le 1, \,\,\forall_{i,j \in S}, \forall_{t,u \ge 0} \quad \text{ and }\quad \sum_{j \in S}p_{ij}(t,u)=1,\,\,,\forall_{i \in S}, \forall_{t,u \ge 0}
\end{equation}
\end{Rem}
\begin{Rem}\normalfont
Let $t,u \geq 0$ and $j \in S$, then,
\begin{align}
\pi_j(t,u) & := P \left\{ X(t,u)=j \right\}\notag\\
& = \sum_{i \in S} P\left\{ X(t,u)=j,\,\,X(0,0)=i\right\}\notag \\
& = \sum_{i \in S} P\left\{ X(t,u)=j\mid X(0,0)=i\right\}\cdot P\left\{ X(0,0)=i\right\} \notag \\
\text{So, \quad \quad} \pi_j(t,u) & =  \sum_{i \in S} p_{ij}(t,u)\cdot \pi_i(0,0)\label{c4Resp_j}
\end{align}
Also,{\begin{equation}
 \sum_{j \in S} \pi_j(t,u) =  1
\end{equation}}
\end{Rem}
\begin{Df}[The transition probability matrix]\normalfont
Let $t,u \ge 0$, let us definite the transition probability matrix
by $\mathbf{P} (t,u) =\Big( p_{ij}(t,u)\Big)_{i,j \in S}$ \,\, and the probability row vector by \,\,
 \mbox{\Large $\mathbf{\pi}$}$(t,u)=\Big( \pi_{j}(t,u)\Big)_{j \in
 S}$. Then, \eqref{c4Resp_j}  can be written as:
\begin{equation}\mbox{\Large $\mathbf{\pi}$}(t,u)=\mbox{\Large
$\mathbf{\pi}$}(0,0)\cdot \mathbf{P}(t,u).\end{equation}
\end{Df}
\begin{Rem}\normalfont
Note that $\mathbf{P}(0,0)=\mathbf{I}$, where $\mathbf{I}$ is of identity matrix.
\end{Rem}
\begin{Df}[Initial probability vector]\normalfont The vector \mbox{\Large $\mathbf{\pi}$}$(0,0)$ is called the initial probability vector.\end{Df}

\section{The waiting region for a change of state}
\setcounter{equation}{0}
\begin{Df}[The waiting region]\normalfont
Suppose that $\{X(t,u)\}$ is a MCHCTP and that
in the point $(t_0,u_0)=(0,0)$, the state of the process
$X(t_0,u_0)=X(0,0)=i$, is known. The time-use taken for a change
of state from state $i$, is a random vector, say
$(\tau_i,\gamma_i)$ which is called the \textit{waiting region for
a change of state from state } $i$.
\end{Df}
\begin{No}[cdf]\normalfont
If $i \in S$, and $t,u\ge 0$, we write the cumulative distribution function of the waiting region $(\tau_i,\gamma_i)$ as:
$$F_i(t,u):=P\big(\tau_i\le t, \gamma_i\le u\mid X(0,0)=0\big).$$
Also, we write
$$\overline{F}_i(t,u):=P\big(\tau_i> t, \gamma_i> u\mid X(0,0)=0\big).$$
\end{No}
\begin{Prp}\label{c4LM1}
Suppose that {X(t; u)} is a MCHCTP. If $i \in S$, and $s,t,w,u \ge 0$, then,
\begin{equation}\label{c4FunctionRProperty}
\overline{F}_i(s+t,w+u)=\overline{F}_i(t,u)\cdot \overline{F}_i(s,w)
\end{equation}
\end{Prp}
\begin{proof}
First of all, let us realize that
\begin{multline}\label{c4PrimeraWait}
 P \left\{ \tau_i > s+t, \gamma_i >w+u \mid X(0,0)=i\right\}\\
= P \left\{ \tau_i > s+t, \gamma_i >w+u \mid X(0,0)=i, \tau_i > s, \gamma_i > w\right\}\cdot\\ P \left\{ \tau_i > s, \gamma_i >w \mid X(0,0)=i\right\}
%=& \mbox{\footnotesize $P \left\{ \tau_i > s+t, \gamma_i >w+u \mid X(0,0)=i, \tau_i > s, \gamma_i > w\right\}\cdot P \left\{ \tau_i > s, \gamma_i >w , X(0,0)=i\right\}\cdot  P \left\{ X(0,0)=i\right\}$}\\
\end{multline}
Now, since $\{X(t,u)\}$ is homogeneous, then
\begin{align}\label{c4SegundaWait}
\begin{split}
 & P \left\{ \tau_i > s+t, \gamma_i >w+u \mid X(0,0)=i, \tau_i > s, \gamma_i > w\right\} \\
 = &  P \left\{ \tau_i > s+t, \gamma_i >w+u \mid X(s,w)=i\right\}\\
 = &  P \left\{ \tau_i > s+t, \gamma_i >w+u , X(s+t,w+u)=i\mid X(s,w)=i \right\}\\
  = &  P \left\{ \tau_i > t, \gamma_i >u , X(t,u)=i\mid X(0,0)=i \right\}\\
  = &  P \left\{ \tau_i > t, \gamma_i >u \mid X(0,0)=i\right\}
  \end{split}
\end{align}

then by \eqref{c4SegundaWait}, the equality \eqref{c4PrimeraWait} can be written as \eqref{c4FunctionRProperty}.
\end{proof}
\begin{Rem}\label{c4CHexp}\normalfont
In the paper \cite{ma} of Marshall and Olkin, we may see that if $X$ and $Y$ are two random variables, such that \begin{equation}\label{c4LM}P(X>s+t,Y>w+u)=P(X>t,Y>u)\cdot P(X>s,Y>w)\,\end{equation}  for all $s,t,w,u \ge 0$,
then:
\begin{equation}
P(X>t,Y>u)= \exp \Big\{-\lambda_1\,t-\lambda_2\,u\Big\},\qquad u,t\ge 0
\end{equation}
for some $\lambda_1,\,\lambda_2\, > 0$.
\end{Rem}
\begin{proof}
In univariate distribution theory is known that if $Y$ is a positive random variable, then $Y\sim \exp(\lambda)$, for some $\lambda >0$, is equivalent to \begin{equation}\label{c4CharExp}\overline{F}_Y(s + t) = \overline{F}_Y(s)\,\overline{F}_Y(t), \text{ for all } s,t\ge 0,\end{equation}where $\overline{F}_Y(s):=P\big(Y > s\big)$.

Now, let $s,t,w,u \ge 0$ and suppose \eqref{c4LM}, then:
\begin{align}\label{c4RX}
\begin{split}
\overline{F}(s,0) & = P\big(X > s, Y >0\big)\\
& = P(X>s)\\
& = \overline{F}_X(s)\end{split}
\end{align}
So,
\begin{align*}
\overline{F}_X(s+t) & =\overline{F}(s+t,0)\quad &\text{ by \eqref{c4RX}},\\
& =\overline{F}(s,0)\,\overline{F}(t,0)\quad & \text{ by \eqref{c4LM}},\\
& =\overline{F}_X(s)\,\overline{F}_X(t)\quad & \text{ by \eqref{c4RX}}.\\
\end{align*}
Therefore, $X\sim \exp (\lambda_1)$, for some $\lambda_1 > 0$. Similarly,
 $\overline{F}(0,u)=\overline{F}_Y(u)$,
$\overline{F}_Y(w+u) =\overline{F}_Y(w)\,\overline{F}_Y(u)$ \quad and $Y\sim \exp (\lambda_2)$,\quad for some  $\lambda_2 > 0$.
Finally,
\begin{align*}
\overline{F}(t,u) & = \overline{F}(t+0,0+u)\\
& = \overline{F}(t,0)\,\overline{F}(0,u) \quad &\text{ by \eqref{c4LM}},\\
& = \overline{F}_X(t)\, \overline{F}_Y(u)\\
& = e^{-\lambda_1\,t - \lambda_2\,u}
\end{align*}
\end{proof}

\begin{Prp}
Suppose that {X(t; u)} is a MCHCTP. If $i \in S$, then
\begin{equation}\label{c4FunctionRProperty1}
\overline{F}_i(t,u)= \exp \Big\{-\lambda_{1i}\,t-\lambda_{2i}\,u\Big\},\qquad t,u\ge 0
\end{equation}
for some $\lambda_{1i},\,\lambda_{2i}\,\ge 0$.
\end{Prp}
By using the proposition \ref{c4LM1} and the remark \ref{c4CHexp}, the result is immediately.
\section{Kolmogorov Equations}
\setcounter{equation}{0}
Next it is going to enunciate and demonstrate the important result of Chapman-Kolmogorov, for the case of Markov chains with two parameters.
\begin{Th}[Chapman Kolmogorov equations]
If $\{X(t,u)\}$ is a  MCHCTP and let be $i,j \in S$, $t,s,w,u \ge 0$, then:
\begin{equation}\label{c4CKe}
p_{ij}(t+s,u+w)=\sum_{k \in S} p_{ik}(t,u)\cdot p_{kj}(s,w)
\end{equation}
or in  matrix  notation,
\begin{equation}\label{c4CKeMatrix}
\mathbf{P}(t+s,u+w)=\mathbf{P}(t,u)\cdot\mathbf{P}(s,w)
\end{equation}
\end{Th}
\begin{proof} We have that,
\begin{align*}
p_{ij}(s+t,w+u)& =P_r\left\{X(s+t,w+u)=j \mid X(0,0)=i \right\}\\
& =\sum_{k \in S}P_r\left\{X(s+t,w+u)=j , X(t,u)=k\mid X(0,0)=i \right\}\\
& =\sum_{k \in S}\mbox{\scriptsize $P_r\left\{X(s+t,w+u)=j \mid  X(t,u)=k, X(0,0)=i \right\}\cdot P_r\left\{X(t,u)=k\mid  X(0,0)=i \right\}$}\\
& =\sum_{k \in S}\mbox{\footnotesize $P_r\left\{X(s+t,w+u)=j \mid  X(t,u)=k\right\}\cdot P_r\left\{X(t,u)=k\mid  X(0,0)=i \right\}$ }\\
 & \text{ \qquad (because $\{X(t,u)\}$ is a Markov chain)}\\
 & =\sum_{k \in S}\mbox{\footnotesize $P_r\left\{X(s,w)=j \mid  X(0,0)=k\right\}\cdot P_r\left\{X(t,u)=k\mid  X(0,0)=i \right\}$ }\\
  & \text{ \qquad (because $\{X(t,u)\}$ is a homogeneous chain)}
\end{align*}
So \eqref{c4CKe} is obtained  and \eqref{c4CKeMatrix} results then immediatly.
\end{proof}
\begin{Df}[Infinitesimal transition between states]\label{c4DefTransitionDensity} \normalfont
If we suppose that for $t=0$ or $u=0$, $\mathbf{P}(t,u)=\mathbf{I}
$, then we define \textit{the infinitesimal transition from state
$i$ to state $j$}, as \,\,
$
a_{ij}=\frac{\,\,\partial\,^2 \,p_{ij}\,\,}{\partial t\, \partial u}\,(0,0)
$. Also, it is defined  \textit{the infinitesimal transition  matrix
} as the matrix $\mathbf{A}=\left(a_{ij}\right)_{i,j \in S}$.
\end{Df}

\begin{Rem}\normalfont
Note that:
\begin{align*}
a_{ij} & =  \frac{\,\,\partial\,^2 \,p_{ij}\,\,}{\partial t\, \partial u}\,(0,0)\\[3mm]
&= \lim_{\overset{\scriptstyle h \to 0^+}{\scriptstyle k \to 0^+}}\frac{1}{h\,k}\big[p_{ij}(h,k)-p_{ij}(h,0)-p_{ij}(0,k)+p_{ij}(0,0)\big]\\[3mm]
\text{So, if } i \ne j, \quad \quad  a_{ij}&=\lim_{\overset{\scriptstyle h \to 0^+}{\scriptstyle k \to 0^+}}\,\frac{p_{ij}(h,k)}{h\,k}\quad \text{and }\quad  a_{ii}=\lim_{\overset{\scriptstyle h \to 0^+}{\scriptstyle k \to 0^+}}\,\frac{p_{ii}(h,k)-1}{h\,k}\label{c42resTransDen}
\end{align*}
These relations show that $a_{ij} \ge 0$ if $i \ne j$, and than $a_{ii} \le 0$.
Therefore,
\begin{align*}
\text{If } i \ne j, \quad \quad & p_{ij}(h,k)=h\,k\,a_{ij}+o(h)\,o(k)\\
\text{and }\quad \quad \quad & p_{ii}(h,k)=1+h\,k\,a_{ij}+o(h)\,o(k)
\end{align*}
Moreover:
\begin{align*}
\sum_{j \in S} a_{ij} & = \sum_{j \in S} \left[ \,\lim_{\overset{\scriptstyle h \to 0^+}{\scriptstyle k \to 0^+}}\frac{p_{ij}(h,k)-p_{ij}(h,0)-p_{ij}(0,k)+p_{ij}(0,0)}{h\,k}\,\right]\\
& = \lim_{\overset{\scriptstyle h \to 0^+}{\scriptstyle k \to 0^+}}\,\frac{1}{h\,k}\,\left[ \, \sum_{j \in S} p_{ij}(h,k)\, -\sum_{j \in S}p_{ij}(h,0)-\, \sum_{j \in S}p_{ij}(0,k)+\, \sum_{j \in S}p_{ij}(0,0)  \,\right]\\
& = \lim_{\overset{\scriptstyle h \to 0^+}{\scriptstyle k \to 0^+}}\,\frac{1}{h\,k}\,\left[ \,1\,-\,1\,-\,1\,+\,1\,\right]=0 \text{ , by \eqref{c4FProb}}
\end{align*}
So,
\begin{equation}\label{c4ResTranDenaii}
a_{ii}=-\sum_{{\overset{\scriptstyle j \in S}{\scriptstyle j \ne i }}}\,a_{ij}, \quad \forall_{ i \in S}.
\end{equation}
\end{Rem}

\begin{Th}[Backward Kolmogorov equations]
If $\{X(t,u)\}$ is a MCHCTP, $i,j \in S$, $t,u \ge 0$, $\mathbf{P}(t,u)=\Big(
p_{ij}(t,u) \Big)_{i,j \in S}$ is the transition probability
matrix and $\mathbf{A}=\Big( a_{ij} \Big)_{i,j \in S}$ is the
infinitesimal transition  matrix, then:
\begin{align}
\frac{\partial^2\,p_{ij}(t,u)}{\partial t\,\partial u\,}& =\sum_{k \in S}\,a_{ik}\,p_{kj}(t,u),\\
\intertext{ or in matrix form:}
\frac{\partial^2\,\mathbf{P}(t,u)}{\partial t\,\partial u\,}&=\mathbf{A}\cdot \mathbf{P}(t,u).\label{c4BKe}
\end{align}
\end{Th}
\begin{proof}
Recall the Chapman Kolmogorov equation \eqref{c4CKe}, we have that:
\[p_{ij}(t+s,u+w)=\sum_{k \in S} p_{ik}(t,u)\cdot p_{kj}(s,w)\]
If $x(t,s)=t+s$ and $y(u,w)=u+w$; and differenting both sides of the before equation with respect to $t$, it is obtained:
 $$\frac{\partial \,p_{ij}}{\partial \,x}\,(x,y)\cdot \frac{\partial \,x }{\partial \,t }\,(t,s) +  \frac{\partial \,p_{ij}}{\partial \,y}\,(x,y)\cdot \frac{\partial \,y }{\partial \,t }\,(u,w)  =\sum_{k \in S}\frac{\partial \, p_{ik}}{\partial \,t }\,(t,u) \cdot p_{kj}\,(s,w)$$
So, $$\frac{\partial \,p_{ij}}{\partial \,x}\,(x,y)  = \sum_{k \in S}\frac{\partial \, p_{ik}}{\partial \,t }\,(t,u) \cdot p_{kj}\,(s,w)
$$
 Now differenting both sides of the before equation with respect to $u$, it is obtained:
 \small
$$
 \frac{\partial^2 \,p_{ij}}{\partial \,x\,\partial \,x}\,(x,y)\cdot \frac{\partial \,x }{\partial \,u }\,(t,s) +  \frac{\partial ^2\,p_{ij}}{\partial \,y\,\partial \,x}\,(x,y)\cdot \frac{\partial \,y }{\partial \,u}\,(u,w) =\sum_{k \in S}\frac{\partial^2 \, p_{ik}}{\partial \,u\,\partial \,t }\,(t,u) \cdot p_{kj}(s,w)$$\normalsize
 Therefore,$$\frac{\partial ^2\,p_{ij}}{\partial \,y\,\partial \,x}(x,y)  =\sum_{k \in S}\frac{\partial^2 \, p_{ik}}{\partial \,u\,\partial \,t }\,(t,u) \cdot p_{kj}(s,w)
$$
But, taking $t=0$ and $u=0$, then
  \begin{align*}
  \frac{\partial ^2\,p_{ij}}{\partial \,w\,\partial \,s}(s,w) & =\sum_{k \in S}\frac{\partial^2 \, p_{ik}}{\partial \,u\,\partial \,t }\,(0,0) \cdot p_{kj}(s,w)\\
   & = \sum_{k \in S}\,a_{ik} \cdot p_{kj}(s,w)
 \end{align*}
 and thus the theorem has been proved.
\end{proof}
$\\$
Similarly can be proved the next theorem:
\begin{Th}[Forward Kolmogorov equations]
If $\{X(t,u)\}$ is a MCHCTP, $i,j \in S$, $t,u \ge 0$, $\mathbf{P}(t,u)=\Big(
p_{ij}(t,u) \Big)_{i,j \in S}$ is the transition probability
matrix and $\mathbf{A}=\Big( a_{ij} \Big)_{i,j \in S}$ is the
infinitesimal transition  matrix, then
\begin{align}\label{c4FKe}
\frac{\partial^2\,p_{ij}(t,u)}{\partial t\,\partial u\,}&=\sum_{k \in S}\,p_{ik}(t,u)\,a_{kj},\\
\intertext{ or in a matrix form:}
\frac{\partial^2\,\mathbf{P}(t,u)}{\partial t\,\partial u\,}&=\mathbf{P}(t,u)\cdot\mathbf{A}
\end{align}
\end{Th}
\begin{No}[Laplace transform.]\normalfont
If $g(x), \,x\ge 0,$ is any function for  which exists its Laplace Transform, then we shall write
$
g^{\ast}(s)=\mathcal{L}(g(x))=\int_0^{\infty} \mbox{\large $\textit{e}$}^{\,-sx}\ g(x)\ dx\quad \text{if} \quad s\ge 0
$
for the \textit{Laplace Transform of} \ $g(x)$. Similarly, if $\mathbf{g}(x)=\Big( g_{ij}(x) \Big)$ is a matrix function for which all its components have Laplace Transform, then we shall write
$
\mathbf{ g}^{\ast}(s)=\mathcal{L}(\mathbf{g}(x))=\Big( {g}^*_{ij}(s) \Big)
$
for the \textit{Laplace Transform of} \ $\mathbf{g} (x)$.\\
If $k(x,y), \,x\ge 0,\, y\ge 0$ is any bivariate function for  which exists its bivariate Laplace Transform, then we shall write
$
k^{\ast \ast}(s_1,s_2)=\mathcal{L}^2(k(x,y))=\int_0^{\infty} \int_0^{\infty} \mbox{\large $\textit{e}$}^{\,-s_1\,x-s_2\,y}\ k(x,y)\ dx\ dy\quad \text{if} \quad s_1,s_2\ge 0
$
for the \textit{bivariate Laplace Transform of} \ $k(x,y)$. Similarly, if $\mathbf{k}(x,y)=\Big( k_{ij}(x,y) \Big)$ is a matrix function for which all its components have Laplace Transform, then we shall write
$
\mathbf{k}^{\ast \ast}(s_1,s_2)=\mathcal{L}^2(\mathbf{k}(x,y)))=\Big( {k}^{\ast \ast}_{ij}(s_1,s_2) \Big)
$
for the \textit{bivariate Laplace Transform of} \ $\mathbf{k} (x,y)$.
\end{No}
Now, we are going to give a relation, between $\mathbf{A}$ and $\mathbf{P}$, by using the double Laplace transform.
\begin{Th}
If $\{X(t,u)\}$ is a MCHCTP, and let be  $s_1,s_2\ge0$, $\mathbf{P}(t,u)=\Big(
p_{ij}(t,u) \Big)_{i,j \in S}$ the transition probability matrix
and $\mathbf{A}=\Big( a_{ij} \Big)_{i,j \in S}$  the infinitesimal
transition matrix, then
\begin{equation}\label{c4KELaplaceT}
\mathbf{P}^{\ast \ast}(s_1,s_2)=\Big(s_1\,s_2\,\mathbf{I}-\mathbf{A}\Big)^{-1}
\end{equation}
where $\mathbf{I}$ is the identity matrix.
\end{Th}
\begin{proof}
Let be $t,u \ge 0$. Remember that in definition \ref{c4DefTransitionDensity}  we had supposed that $\mathbf{P}(t,u)=\mathbf{I}$ if $t=0$ or $u=0$. Now, let us call $\mathbf{H}(t,u)=\frac{\partial \mathbf{P}}{\partial u}(t,u)$. then \eqref{c4BKe} can be written as
\[  \frac{\partial \mathbf{H}}{\partial t}(t,u)= \mathbf{A}\cdot \mathbf{P}(t,u)\]
and taking Laplace Transform in both sides, with respect to variable $t$,
\begin{equation}\label{c4Auxiliar1}
 s_1\,\mathbf{H}^*(s_1,u)-\mathbf{H}(0,u)=\mathbf{A}\cdot \mathbf{P}^*(s_1,u)%\tag{A}
\end{equation}
But,
\begin{align*}
H(0,u) = \, \frac{\partial \mathbf{P}}{\partial u}\,(0,u) &= \,\lim_{k \to 0}\,\frac{1}{k}\Big[\,\mathbf{P}(0,u+k)\,-\,\mathbf{P}(0,u)\,\Big]\\[4mm]
&= \lim_{k \to 0}\,\frac{1}{k}\Big[\,\mathbf{I}\,-\,\mathbf{I}\,\Big]=\mathbf{0}
\end{align*}
where $\mathbf{0}$ is the null matrix. So \eqref{c4Auxiliar1} can be written as:
\begin{equation}\label{c4Auxiliar2}
 s_1\,\mathbf{H}^*(s_1,u)=\mathbf{A}\cdot \mathbf{P}^*(s_1,u)%\tag{B}
\end{equation}
Now, we observe that:
\begin{align*}
\mathbf{H}^*(s_1,u) & = \int_0^{\infty}\,e^{-t\,s_1}\,\mathbf{H}(t,u)\,dt =  \int_0^{\infty}\,e^{-t\,s_1}\,\frac{\partial}{\partial\,u}\mathbf{P}(t,u)\,dt\\[4mm]
 & = \frac{\partial}{\partial\,u}\, \int_0^{\infty}\,e^{-t\,s_1}\,\mathbf{P}(t,u)\,dt = \frac{\partial}{\partial\,u}\,\mathbf{P}^*(s_1,u)
\end{align*}
Then \eqref{c4Auxiliar2} can be written as:
\[s_1\frac{\partial}{\partial\,u}\,\mathbf{P}^*(s_1,u)=\mathbf{A}\cdot \mathbf{P}^*(s_1,u)\]
And taking Laplace Transform in both sides with respect to variable $u$:
\begin{equation}\label{c4Auxiliar3}
s_1\,\Big[ \,s_2\,\mathbf{P}^{\ast \ast}(s_1, s_2)- \mathbf{P}^{\ast}(s_1,0)\,\Big]=\mathbf{A}\cdot \mathbf{P}^{\ast *}(s_1,s_2)%\tag{C}
\end{equation}
But,
\begin{align*}
\mathbf{P}^{\ast}(s_1,0) & = \int_0^{\infty} e^{-t\,s_1}\,\mathbf{P}(t,0)\,dt = \int_0^{\infty} e^{-t\,s_1}\,\mathbf{I}\,dt\\
& = \left(\,\int_0^{\infty} e^{-t\,s_1}\,dt\,\right)\cdot \mathbf{I}=\frac{1}{s_1}\cdot \mathbf{I}
\end{align*}
So \eqref{c4Auxiliar3} can be written as:
\[s_1\,\Big[ \,s_2\,\mathbf{P}^{\ast \ast}(s_1, s_2)- \frac{1}{s_1}\cdot \mathbf{I}\,\Big]=\mathbf{A}\cdot \mathbf{P}^{\ast *}(s_1,s_2)\]
and solving this equation for $\mathbf{P}^{\ast \ast}(s_1, s_2)$, we obtain \eqref{c4KELaplaceT}.
\end{proof}
\section{Application}
\setcounter{equation}{0}
A high technology machine to produce juice has several identical components. It can work if at least one of its components is in a good condition. However, just damaged one of its components, it is removed and repaired. When one or more of its components are being repaired, the machine does not allow a full work. Once it is fixed, it is placed back into the machine. In this sense, the probability that all the components simultaneously are damaged, is practically zero. Moreover, not all the time, the machine has the same amount of work. Therefore, the machine has a measurer that records the amount of work done by the machine. For this reason, the warranty policy takes into account both the time since the machine is running, and the amount of work done. For the model that is being analyzed, we suppose that $X(t, u)$ is a MCHCTP and that its states space is $S =\{0, 1\}$.

It is interpreted as follows: the state 1 means that the
machine is working full capacity, that is, all its components are in good condition. The
state 0 means that at least one of the components is being repaired and the machine is
not working full capacity. The parameter t represents the total time (in years) since the
machine was started. The parameter u represents the amount of work (in million of juice liters) that the machine has performed.\\[4mm]
The infinitesimal transition matrix of the process $\{X(t,u)\}$ is $$\mathbf{A}=\begin{bmatrix}
\quad  -2 \quad &  2 \\
\quad \, \, 0.6 \quad &  -0.6 \quad
\end{bmatrix}.$$
\\
Therefore, by \eqref{c4KELaplaceT},\\ $$\mathbf{P}^{\ast \ast}(s_1,s_2)= \frac{1}{s_1s_2(5s_1s_2+13)}\cdot\begin{bmatrix}
5s_1s_2+3 & 10\\
3 & 5(s_1s_2+2)
 \end{bmatrix},$$
$\\$
So, by using the method for to invert the double Laplace Transform that was showed in \cite{mo} by Moorthy, for instance,  we have found the next results, which have relative errors less that the 4\% :\\[4mm]
$$P(0.2,0.6)= \begin{bmatrix}
    \quad 0.7781 \quad   & \quad 0.2219   \quad   \\
    0.0666   & 0.9334
\end{bmatrix}$$
and

$$P(2.0,2.0)= \begin{bmatrix}
    \quad 0.4272\quad & \quad 0.5754 \quad \\
 0.1726 & 0.8300
\end{bmatrix}.$$
In this application, we suppose that the initial probability vector is
 \mbox{\Large $\mathbf{\pi}$}$(0,0)=$ $\big[\,0,\,1\,\big] $
 and then, we obtain that,
 $$ \text{\Large $\mathbf{\pi}$}(0.2,0.6)=\big[\,0.0666,\,0.9334\,\big],$$
 and
 $$\text{\Large $\mathbf{\pi}$}(2.0,2.0)=\big[\,0.1726,\,0.8300\,\big]. $$

Suppose that the state of the process when $t=0$ and $u=0$  is $X(0,0)=i$. Let $(\tau_i,\gamma_i)$ the waiting region for a change of state from state $i$ and let:
\begin{align*}
F(t,u)& =Pr\left(\tau_0 \le t,\,\gamma_0 \le u\right)\\[2mm]
G(t,u)& =Pr\left(\tau_1 \le t,\,\gamma_1 \le u\right)\\[2mm] \overline{F}(t,u)&=Pr\left(\tau_0>t,\,\gamma_0>u\right)\\[2mm] \overline{G}(t,u)&=Pr\left(\tau_1>t,\,\gamma_1>u\right)\\[2mm]
f(t,u)&=\frac{\partial^2 F}{\partial t \partial u}(t,u)
\intertext{and }
g(t,u)&=\frac{\partial^2 G}{\partial t \partial u}(t,u).
\end{align*}
Then, we can write the next integral equations:
\begin{align}
p_{00}(t,u)=\,&\, \overline{F}(t,u)+\int^u_0\int^t_0 \,p_{10}(t-\xi,u-\omega)\,dF(\xi,\omega)\\[3mm]
p_{01}(t,u)=& \hspace{1.8cm}\int^u_0\int^t_0 \,p_{11}(t-\xi,u-\omega)\,dF(\xi,\omega) \label{c4imp1} \\[3mm]
p_{10}(t,u)=& \hspace{1.8cm}\int^u_0\int^t_0 \,p_{00}(t-\xi,u-\omega)\,dG(\xi,\omega) \label{c4imp2} \\[3mm]
p_{11}(t,u)=\,&\, \overline{G}(t,u)+\int^u_0\int^t_0 \,p_{01}(t-\xi,u-\omega)\,dG(\xi,\omega)
\end{align}
By taking double Laplace Transform on \eqref{c4imp1} and \eqref{c4imp2}, we obtain:
\begin{align}
p_{01}^{\ast \ast}(s_1,s_2)=& \, f^{\ast \ast}(s_1,s_2)\,p_{11}^{\ast \ast}(s_1,s_2)\label{c4imp3}\\[3mm]
p_{10}^{\ast \ast}(s_1,s_2)=& \, g^{\ast \ast}(s_1,s_2)\,p_{00}^{\ast \ast}(s_1,s_2)\label{c4imp4}
\end{align}
Since   $\mathbf{P}^{\ast \ast}(s_1,s_2)$ is known already, then  by \eqref{c4imp3} and \eqref{c4imp4}, we deduce that $$f^{\ast \ast}(s_1,s_2)=10/\left(5s_1s_2+10\right) \,\text{ and }\, g^{\ast \ast}(s_1,s_2)=3/\left(5s_1s_2+3\right)$$ and so $$F^{\ast \ast}(s_1,s_2)=10/\left(s_1s_2(5s_1s_2+10)\right) \, \text{ and }\, G^{\ast \ast}(s_1,s_2)=3/\left(s_1s_2(5s_1s_2+3)\right).$$ $\\$
Again, we use \cite{mo} for to invert these double Laplace transforms.\\[4mm]
The warranty conditions are as follows: If damage in one of  the machine components occurs, within the first six months after putting the machine in use and before it produces $200,000$ liters of juice, the machine provider changes the item immediately by a new one whose components have been fully checked prior to installation and they meet the quality requirements (ie better than new). However, if damage occurs outside the above ranges, but during the first year of operation of the machine and before the machine produces $300,000$ liters of juice, the machine provider agrees to make the change of the just damage component immediately and to make a general revision of the machine to state it better  than new. Once one of these  has been done, the machine provider does not offer more warranty service. Suppose that the cost of the machine is $C$ and the cost of the change of one of the components and the general revision is $1=10\,C$. Then, by using the results proposed by Dimitrov et al \cite{di}, the expected warranty expense is:
\begin{align*}
EWE & \,=  C \cdot G(0.5,0.2) + \frac{C}{10}\cdot \left[G(1,0.3)-G(0.5,0.2)\right]\\[4mm]
 & \,= 0.0591\, C + 0.1130 \,\frac{C}{10}= 0.0704\, C
 \end{align*}

\section{Conclusion}
It is known that the Markov chains with a single continuous parameter are characterized by the infinitesimal transition matrix. This study concluded that Markov chains with two continuous parameters are characterized by a matrix, that it's also called "infinitesimal transition matrix", although its structure is different to the case of a single parameter. In addition, this Markov process can be used in a particular type of two-dimensional warranty policies problems.

\end{document}